\documentclass[10pt, xcolor]{article}
\usepackage{amsfonts}
\usepackage{amssymb,amsthm, amsmath,graphicx,bm,amsthm}
\usepackage{mathrsfs}
\usepackage{color,enumerate,lineno}
\usepackage{url,hyperref}
\usepackage{amsfonts}
\usepackage{amssymb,amsthm, amsmath,graphicx,bm,amsthm,mathtools}
\usepackage{color,enumerate}
\usepackage{url,lineno}
\usepackage{lscape}
\usepackage{multirow}
\usepackage{float}
\usepackage{fancyhdr}
\usepackage{emptypage}

\usepackage[latin1]{inputenc}

\usepackage[all]{xy}

\newtheorem{theorem}{Theorem}[section]

\newtheorem{proposition}[theorem]{Proposition}
\newtheorem{corollary}[theorem]{Corollary}
\newtheorem{lemma}[theorem]{Lemma}
\newtheorem{algorithm}[theorem]{Algorithm}
\theoremstyle{definition}

\newtheorem{example}[theorem]{Example}
\newtheorem{remark}[theorem]{Remark}

\renewcommand{\L}{\mathcal{L}}

\def\CC{\mathscr{C}}

\def\d{{\mathrm a}}

\def\e{{\mathrm e}}
\def\g{{\mathrm g}}

\def\m{{\mathrm m}}
\def\n{{\mathrm n}}
\def\t{{\mathrm t}}

\def\gcd{{\mathrm{gcd}}}

\def\A{{\mathrm{A}}}
\def\B{{\mathrm{B}}}
\def\C{{\mathrm C}}

\def\F{{\mathrm F}}
\def\G{{\mathrm G}}

\def\L{{\mathrm L}}

\def\T{{\mathrm T}}

\def\SG{\mathrm{SG}}
\def\NN{{\mathrm N}}
\def\PF{{\mathrm{PF}}}
\def\Sat{{\mathrm{Sat}}}
\def\msg{{\mathrm{ msg }}}

\def\Ap{{\mathrm{ Ap}}}
\def\max{{\mathrm{ max}}}
\def\MED{{\mathrm{MED}}}
\def\Cad{{\mathrm {Cad}}}

\def\sA{\mathscr{A}}

\def\msg{{\mathrm{ msg }}}

\def\Maximals{\mathrm{Maximals}_{\leq_S}}
\def\max{\mathrm{max}}

\def\min{\mathrm{min}}
\def\max{\mathrm{max}}

\def\N{\mathbb{N}}

\def\Z{\mathbb{Z}}
\def\Q{\mathbb{Q}}
\def\d{\mathrm{d}}

\def\rank{\mathrm{rank}\, }
\def\Ap{\mathrm{Ap}}
\def\int{\mathrm{int}}

\title{The covariety of saturated numerical semigroups with fixed  Frobenius number}

\author{
	M. A. Moreno-Fr\'{\i}as \footnote{
		Dpto. de Matem\'aticas, Facultad de Ciencias,
		Universidad de C\'adiz, E-11510, Puerto Real  (C\'{a}diz, Spain).
		Partially supported by  Junta de Andaluc\'{\i}a group FQM-298, 
		Proyecto de Excelencia de la Junta de Andalucía ProyExcel\_00868, Proyecto de investigación del Plan Propio--UCA 2022-2023 (PR2022-011) and Proyecto de investigación del Plan Propio--UCA 2022-2023 (PR2022-004).
		E-mail: mariangeles.moreno@uca.es.}
	\and
	J. C. Rosales \footnote{
		Dpto. de \'Algebra, Facultad de Ciencias, Universidad de Granada,
		E-18071, Granada. (Spain).
		Partially supported by  Junta de Andaluc\'{\i}a group FQM-343,
		Proyecto de Excelencia de la Junta de Andalucía ProyExcel\_00868 and Proyecto de investigación del Plan Propio--UCA 2022-2023 (PR2022-011).
		E-mail: jrosales@ugr.es.}
}

\date{}

\begin{document}
 
\maketitle

\begin{abstract}

In this work we will show that  if $F$ is  a positive integer, then  $\Sat(F)=\{S\mid S \mbox{ is a saturated numerical semigroup with Frobenius number } F\}$ is a covariety. As a consequence, we present two algorithms: one that computes $\Sat(F),$ and the other  which computes all the elements of $\Sat(F)$ with a fixed genus. 
	 
	  If $X\subseteq S\backslash \Delta(F)$ for some $S\in \Sat(F),$ then we will see that there is  the least element of $\Sat(F)$ containing a $X$. This element will  denote  by $\Sat(F)[X].$
	  If $S\in\Sat(F),$ then we define the $\Sat(F)$-rank of $S$ as the minimum of  $\{\mbox{cardinality}(X)\mid S=\Sat(F)[X]\}.$ In this paper, also we present an algorithm to compute all the element of $\Sat(F)$ with a given 
	  $\Sat(F)$-rank.

\smallskip
    {\small \emph{Keywords:} Numerical semigroup, covariety,  Frobenius number, genus, saturated numerical semigroup,   algorithm.}

   \smallskip
    {\small \emph{MSC-class:} 20M14 (Primary),  11D07, 13H10 (Secondary).}
\end{abstract}

\section{Introduction}

Let $\N$ the set of nonnegative integers numbers. A {\it numerical
	semigroup} is a subset $S$ of $\N$ which is closed by sum, $0\in S$
and $\N\backslash S$ is finite. The set $\N\backslash S$ is known as the set of {\it gaps} of $S$ and its cardinality, denoted by $\g(S),$ is the {\it genus} of $S.$ The largest integer not belonging to $S$, is known as the {\it Frobenius number} of $S$ and it will be denoted by $\F(S).$

Let $\{n_1<\dots <\n_p\}\subseteq \N$ with $\gcd(n_1,\dots, n_p)=1.$ Then 
$
\langle n_1,\dots, n_p \rangle =\left \{\sum_{i=1}^p \lambda_in_i \mid \{\lambda_1,\dots, \lambda_p\}\subseteq \N \right\}
$
is a numerical semigroup and every numerical semigroup has this form (see \cite[Lemma 2.1]{libro}). The set ${n_1<\dots <n_p}$ is called {\it system of generators} of $S$, and we write $S=\langle  n_1,\dots, n_p \rangle.$ 
We say that a system of generators of a numerical semigroup is a {\it minimal system of generators} if none of its proper subsets generates the numerical semigroup. Every numerical semigroup has a unique minimal system of generators, which in addition is finite (see  \cite[Corollary 2.8]{libro}). The minimal system of generators of a numerical semigroup $S$ is denoted by $\msg(S).$ Its cardinality is called the {\it embedding dimension} and will be denoted by $\e(S).$ Another invariant which we will use in this work is the {\it multiplicity} of $S$, denoted by $\m(S).$ It is defined as the minimum of $S\backslash \{0\}.$ 

Given $S$  a numerical semigroup the mulplicity, the genus and the Frobenius number of $S$  are  three invariants  very important in the theory numerical semigroups (see for instance \cite{alfonsin} and \cite{barucci} and the reference given there) and they will play a very important role in this work.

The  Frobenius problem (see \cite{alfonsin}) for numerical semigroups, lies in finding formulas to obtain the Frobenius number and the genus of a numerical semigroup from its minimal system of generators. When the numerical semigroup has embedding dimension two, this problem was solved by J. J. Sylvester (see \cite{sylvester}). However, if the numerical semigroup has embedding dimension greater than or equal to three,  the problem is still open. 

Looking to find a solution to the Frobenius problem, in \cite{covariedades}, we study the set 
 $\sA(F)=\{S\mid S \mbox{ is a numerical semigroup and } \F(S)=F\},$ being $F\in \N\backslash \{0\}.$ The generalization of  $\sA(F)$ as a family of numerical semigroups that verifies certain properties leads us introduce in \cite{covariedades}, the concept of  covariety. That is, a {\it covariety} is  a nonempty family $\CC$ of numerical semigroups that  fulfills  the following conditions:
\begin{enumerate}
	\item[1)]  $\CC$ has a minimum, denoted by $\Delta(\CC)=\min(\CC),$ with respect to set inclusion.
	\item[2)] If $\{S, T\} \subseteq \CC$, then $S \cap  T \in \CC$.
	\item[3)]  If $S \in \CC$  and $S \neq  \Delta(\CC)$, then $S \backslash \{\m(S)\} \in \CC$.
\end{enumerate}

This concept has allowed us to study common properties of some families of numerical semigroups. For instance, in \cite{coarf}, we have studied the set of all numerical semigroups which have the Arf property (see for example, \cite{barucci})  with a given Frobenius number,  showing some algorithms to compute them.

In the semigroup literature one can find a long list of works dedicated to the study of  one dimensional analytically irreducible domains via their value
semigroup (see  for instance \cite{bertin}, \cite{castellanos}, \cite{delorme}, \cite{kunz} and \cite{watanabe}). One of the properties
studied for this kind of rings using this approach has been  to be saturated. Saturated rings were introduced in three different ways by Zariski (\cite{zariski}),
Pham-Teissier (\cite{pham}) and Campillo (\cite{campillo}). Theses tree definitions coincide for
algebraically closed fields of zero characteristic. The characterization of saturated ring 
in terms of their value semigroups gave rise to the notion of saturated numerical semigroup (see [\cite{delgado}, \cite{nunez}]).

If $A \subseteq \N$ and $a\in A,$ then we denote by $\d_A(a)=\gcd\{x\in A\mid x\leq a\}.$ A numerical semigroup $S$ is {\it saturated} if $s+\d_S(s)\in S$ for all $s\in S\backslash \{0\}.$

If $F\in \N\backslash \{0\},$ we denote by 
$$\Sat(F)=\{S\mid S \mbox{ is a saturated numerical semigroup and }\F(S)=S\}.$$
In this work, we study the set of $\Sat(F)$ by using the techniques of covarieties.

 The structure of the paper is the following.  Section 2 will be devoted to recall some concepts and result which will be used in this work. Also, we show how we can compute some of then  with the help  of the \texttt{GAP} \cite{GAP} package \texttt{numericalsgps} \cite{numericalsgps}. In Section 3 we will show  that $\Sat(F)$ is a covariety. This fact allow us to order the elements of $\Sat(F)$ making a rooted tree, and consequently, to present an algorithm which computes all the elements of $\Sat(F).$
 
 In Section 4, we will see who are the maximal elements  of $\Sat(F).$ We compute the set $\{\g(S)\mid S\in \Sat(F)\}$ and we apply this result to give an algorithm which enables to calculate all the elements of $\Sat(F)$ with a fixed genus.
 
 A set $X$ is called a $\Sat(F)$-{\it set}, if it verifies the following conditions:
 \begin{enumerate}
 	\item[1)] $X\cap \{0,F+1,\rightarrow\}=\emptyset,$  where the symbol $\rightarrow$ means that every integer greater than $F+1$ belongs to the set.
 	\item[2)]There is $S\in \Sat(F)$ such that $X\subseteq S.$
 \end{enumerate}

In Section 5, we  will see that if $X$ is a $\Sat(F)$-set, then there is the least element of $\Sat(F)$ containing a $X$. This element will  denote  by $\Sat(F)[X].$

If $S=\Sat(F)[X],$ then we will say that $X$ is a  $\Sat(F)$-{\it system of generators of } $S.$ Also, we will show that every element of $\Sat(F)$ admits a unique minimal $\Sat(F)$-sytem of generators.

The $\Sat(F)$-{\it rank} of an element of $\Sat(F)$ is the cardinal of its minimal $\Sat(F)$-sytem of generators. In Section 6, we presente an algorithmic procedure to compute all the elements of $\Sat(F)$ with a given  $\Sat(F)$-rank.

\section{Preliminaires}
In this section we present some concepts and results which are necessary for the understant of the work.
The following result appears in \cite[Proposition 3.10]{libro}.
\begin{proposition}\label{proposition1}
	If $S$ is a numerical semigroup, then $\e(S)\leq \m(S).$
\end{proposition}

A numerical semigroup S is said to have {\it maximal embedding
	dimension} (from now on $\MED$-{\it semigroup}) if $\e(S) = \m(S).$ 

By applying the results of \cite[Chapter 3]{libro}, we have the following result.

\begin{proposition}\label{proposition2}
	Every saturated numerical semigroup is a $\MED$-semigroup.
\end{proposition}

Following the notation introduced in \cite{JPAA}, an integer $z$ is a {\it pseudo-Frobenius number} of a numerical semigroup $S$ if $z\notin S$ and $z+s\in S$ for all $s\in S\backslash \{0\}.$  We denote by $\PF(S)$
the set formed by the  pseudo-Frobenius numbers of $S.$ The cardinality of $\PF(S)$ is an important invariant of $S$ (see \cite{froberg} and \cite{barucci}) called the {\it type} of $S,$  denoted by $\t(S).$\\

For instance, let $S=\langle 7,8,9,11,13 \rangle$, if  we want to calculate the set $\PF(S)$, then we use the following sentences:
\begin{verbatim}
	gap> S := NumericalSemigroup(7,8,9,11,13);
	<Numerical semigroup with 5 generators>
	gap> PseudoFrobeniusOfNumericalSemigroup(S);
	[ 6, 10, 12 ]
\end{verbatim}

%

Given $S$  a numerical semigroup, we denote by  $\SG(S)=\{x\in \PF(S)\mid 2x \in S\}.$ Its elements will be called {\it special gaps} of $S.$

For instance, given the numerical semigroup $S=\langle 6,7,8,10,11 \rangle$, if  we want to calculate the set $\SG(S)$, then we use the following sentences:
\begin{verbatim}
	gap> S := NumericalSemigroup(6,7,8,10,11);
	<Numerical semigroup with 5 generators>
	gap> SpecialGaps(S);
	[ 4, 5, 9 ]
\end{verbatim}

The following result appears in \cite[Proposition 4.33]{libro}.
\begin{proposition}\label{proposition5} Let $S$ be a numerical semigroup and $x\in \N\backslash S.$ Then $x\in \SG(S)$ if and only if $S \cup \{x\}$ is a numerical semigroup.	
\end{proposition}

Let $S$ be a numerical semigroup and $n\in S\backslash \{0\}.$ Define the 
{\it Apéry set} of $n$ in $S$ (in honour of \cite{apery})  as 
$\Ap(S,n)=\{s\in S\mid s-n \notin S\}$. \\

For instance, to compute $\Ap(S,8),$ being $S=\langle 8,9,11,13 \rangle,$  we use  the following sentences:
\begin{verbatim}
	gap> S := NumericalSemigroup(8,9,11,13);
	<Numerical semigroup with 4 generators>
	gap> AperyList(S,8);
	[ 0, 9, 18, 11, 20, 13, 22, 31 ]
	\end{verbatim}

The following result is deduced from \cite[Lemma 2.4]{libro}.

\begin{proposition}\label{proposition6}
	Let $S$ be  a  numerical semigroup and $n\in S\backslash \{0\}.$ Then $\Ap(S,n)$ is a set with cardinality $n.$ Moreover, $\Ap(S,n)=\{0=w(0),w(1), \dots,\\ w(n-1)\}$, where $w(i)$ is the least
	element of $S$ congruent with $i$ modulo $n$, for all $i\in
	\{0,\dots, n-1\}.$
\end{proposition}
From Proposition 3.1 of \cite{libro}, we can deduce the following result.
\begin{proposition}\label{proposition7}
	Let $S$ be a numerical semigroup. Then $S$ is a $\MED$-semigroup if and only if $\msg(S)=\left(\Ap(S,\m(S))\backslash \{0\} \right)\cup \{\m(S)\}.$
\end{proposition}

Let $S$ be a numerical semigroup. Over $\Z$ we  define the following order relation: $a\leq_S b$ if $b-a \in S.$ The following result is Lemma 10 from \cite{JPAA}.
\begin{proposition}\label{proposition8} If $S$ is  a numerical semigroup and $n \in S\backslash
	\{0\}.$ Then
	$$
	\PF(S)=\{w-n\mid w \in \Maximals \Ap(S,n)\}.
	$$
\end{proposition}
The next lemma  has an immediate proof.

\begin{proposition}\label{proposition9}
	Let $S$ be a numerical semigroup, $n\in S\backslash \{0\}$ and $w \in \Ap(S,n).$ Then $w\in \Maximals(\Ap(S,n))$ if and only if $w+w'\notin
	\Ap(S,n)$ for all $w'\in \Ap(S,n)\backslash \{0\}. $\\	
\end{proposition}

The proof of the following result is very simple.

\begin{proposition}\label{proposition10} If $S$ is a numerical semigroup and $S\neq \N,$ then $$\SG(S)=\{x\in \PF(S)\mid 2x \notin \PF(S)\}.$$
\end{proposition}

\begin{remark}\label{remark11}
	Observe that as a consequence from Propositions \ref{proposition8}, \ref{proposition9} and \ref{proposition10}, if $S$ is a numerical semigroup and we know the set $\Ap(S,n)$ for some $n\in S\backslash \{0\},$ then we easily can calculate the set $\SG(S).$
\end{remark}

The following result is well known as well as it is very easy to prove.

\begin{proposition}\label{proposition12} Let $S$ and $T$ be numerical semigroups and $x\in S.$ Then the following hold:
	\begin{enumerate}
		\item[1)] $S\cap T$ is a numerical semigroup and $\F(S\cap T)=\max\{\F(S), \F(T)\}.$
		\item[2)] $S\backslash \{x\}$ is a numerical semigroup if and only if $x\in \msg(S).$
		\item[3)] $\m(S)=\min\left( \msg(S)\right).$
	\end{enumerate}	
\end{proposition}

The following result is Lemma 2.14 of \cite{libro}.
\begin{proposition}\label{proposition13}
	If $S$ is a numerical semigroup, then $\frac{\F(S)+1}{2}\leq \g(S).$
\end{proposition}

\section{The tree associated to $\Sat(F)$}

Our first aim in this section will be to prove that if $F$ is a positive integer, then the set $\Sat(F)=\{S\mid S \mbox{ is a saturated numerical semigroup and }\F(S)=F\}$ is a covariety.\\

The following result can be consulted in \cite[Proposition 5]{houston1}.

\begin{lemma}\label{lemma14} If $S$ and $T$ are saturated numerical semigroups, then $S\cap T$ is also a saturated numerical semigroup.
\end{lemma}

The following result has an immediate proof.

\begin{lemma}\label{lemma15}Let $F$ be a positive integer. Then  the following properties are verified:
	\begin{enumerate}
		\item[1)] If $m\in \N,$ then $\Delta(m)=\{0,m,\rightarrow\}$, is a saturated numerical semigroup.
		\item[2)]$\Delta(F+1)$ is the minimun of  $\Sat(F).$
		\item[3)] If $S$ is a saturated numerical semigroup, then $S\backslash \{\m(S)\}$ is also a saturated numerical semigroup.
	\end{enumerate}
\end{lemma}

By applying Proposition \ref{proposition12}; and Lemmas \ref{lemma14} and \ref{lemma15}, we  can easily deduce the following result.
\begin{proposition}\label{proposition16}
	If $F$ is a positive integer, then $\Sat(F)$ is a covariety.
\end{proposition}

A {\it graph} $G$ is a pair $(V,E)$ where $V$ is a nonempty set and
$E$ is a subset of $\{(u,v)\in V\times V \mid u\neq v\}$. The
elements of $V$ and $E$ are called {\it vertices} and {\it edges},
respectively. A {\it path, of
	length $n$,} connecting the vertices $x$ and $y$ of $G$ is a
sequence of different edges of the form $(v_0,v_1),
(v_1,v_2),\ldots,(v_{n-1},v_n)$ such that $v_0=x$ and $v_n=y$.

A graph $G$ is {\it a tree} if there exists a vertex $r$ (known as
{\it the root} of $G$) such that for any other vertex $x$ of $G,$
there exists a unique path connecting $x$ and $r$. If  $(u,v)$ is an
edge of the tree $G$, we say that $u$ is a {\it child} of $v$.\\

Define the graph $\G(F)$ as follows:

\begin{itemize}
	\item the set of vertices of $\G(F)$ is $\Sat(F)$,
	\item $(S,T)\in \Sat(F) \times \Sat(F)$ is an edge of $\G(F)$ if and only if $T=S\backslash \{\m(S)\}.$
\end{itemize}

As a consequence from  \cite[Proposition 2.6]{covariedades} and Proposition \ref{proposition16}, we have the following result.
\begin{proposition}\label{proposition17}Let $F$ be a positive integer. Then
	$\G(F)$ is a tree with root $\Delta(F+1).$	
\end{proposition}

A tree can be  built recurrently starting from the root  and connecting, 
through an edge, the vertices already built with  their  children. Hence, it is very interesting to characterize the children of an arbitrary vertex of $\G(F).$ For this reason, we will introduce some concepts and results.\\

The following result is deduced from Proposition \ref{proposition16} and \cite[Proposition 2.9]{covariedades}.
\begin{proposition}\label{proposition18}
	If $S\in \Sat(F),$ then the set formed by the children of $S,$ in the tree $\G(F),$ is the set
	$$
	\{S\cup \{x\}\mid x\in \SG(S),\, x<\m(S) \mbox{ and }S\cup \{x\}\in \Sat(F)\}.
	$$
\end{proposition}	

Let $S\in \Sat(F)$ and $x\in \SG(S)$ such that $x<\m(S)$ and $x\neq F.$ The following result provides us an algorithm to decide if  $S\cup \{x\}$ belongs to $\Sat(F).$

\begin{proposition}\label{proposition19}
Let $S\in \Sat(F),$  $x\in \SG(S)$ such that $x<\m(S)$ and $x\neq F.$ Then $S\cup \{x\}\in \Sat(F)$ if and only if $s+\d_{S\cup \{x\}}(s)\in S$ for every $s\in \{\m(S),\cdots, \m(S)+x\}.$
\end{proposition}
\begin{proof}
	{\it Necessity.} Trivial. \\
{\it Sufficiency.} We have to prove that if $s\in S$ and $s>\m(S)+x,$ then 	$s+\d_{S\cup \{x\}}(s)\in S.$ Hence it is enough to show that $\d_{S}(s)=\d_{S\cup \{x\}}(s).$ But it is true because $\d_{S}(s)=\gcd\{ \m(S),\cdots, \m(S)+x,\cdots,s\}=\gcd\{x, \m(S),\cdots,s\}=\d_{S\cup \{x\}}(s).$
\end{proof}

\begin{example}\label{example20}
	It is clear that $S=\{0,8,10,12,14,16,18,\rightarrow\}\in \Sat(17)$ and $6\in \SG(S).$\\
	 As 
	$$
	\{8+\d_{S\cup \{6\}}(8),10+\d_{S\cup \{6\}}(10),12+\d_{S\cup \{6\}}(12), 14+\d_{S\cup \{6\}}(14) \}=$$
	$$\{8+2,10+2,12+2,14+2\}\subseteq S,
	$$
	then by applying Proposition \ref{proposition19}, we have that $S\cup \{6\}\in \Sat(17).$
\end{example}
The next proposition is  Proposition 4.6 of \cite{coarf}.
\begin{proposition}\label{proposition21} Let $S$ be a numerical semigroup and $x\in \SG(S)$ such that $x<\m(S)$ and $S\cup \{x\}$ is a $\MED$-semigroup. Then the following conditions   hold.
	\begin{enumerate}
		\item[1)] For  every $i\in \{1,\dots, x-1\}$ there is $a\in \msg(S)$ such that $a\equiv i\, (\mbox{mod } x).$
		\item[2)] If $\alpha(i)=\min\{a\in \msg(S)\mid a\equiv i\, (\mbox{mod } x) \}$ for all $i\in \{1,\dots, x-1\},$ then $\msg(S\cup \{x\})=\{x,\alpha(1),\dots,\alpha(x-1)\}.$
	\end{enumerate}
\end{proposition}

\begin{remark}\label{remark22}
		Note that as a consequence of Propositions \ref{proposition2}, \ref{proposition18} and \ref{proposition21}, if $S\in \Sat(F)$ and we know the set $\msg(S),$ then  we can easily compute $\msg(T)$ for every child $T$ of $S$ in the tree $\G(F).$
\end{remark}

\begin{algorithm}\label{algorithm23}\mbox{}\par
\end{algorithm}
\noindent\textsc{Input}: A positive integer $F.$   \par
\noindent\textsc{Output}: $\Sat(F).$

\begin{enumerate}
	\item[(1)] $\Delta=\langle F+1,\dots, 2F+1 \rangle,$ $\Sat(F)=\{\Delta\}$ and  $B=\{\Delta\}.$ 
	\item[(2)] For every $S \in B,$  compute $\theta(S)=\{x\in \SG(S)\mid x<\m(S), x\neq F \mbox{ and } S\cup \{x\} \mbox { is a saturated numerical semigroup}\}$ (by using Proposition \ref{proposition7}, Remark \ref{remark11} and Proposition \ref{proposition19}).
	\item[(3)] If $\displaystyle\bigcup_{S\in B}\theta(S)=\emptyset,$ then return $\Sat(F).$
	\item[(4)]  $C=\displaystyle\bigcup_{S\in B}\{S\cup \{x\}\mid x\in \theta(S)\}.$ 	
	\item[(5)] For all $S\in C$ compute $\msg(S),$ by using Proposition \ref{proposition21}.
	\item[(6)]  $\Sat(F)= \Sat(F)\cup C,$ $B=C,$ and   go to Step $(2).$ 	
\end{enumerate}
Next we illustrate this algorithm with an example.

\begin{example}\label{example33}
	We are going to calculate  $\Sat(7)$, by using Algorithm \ref{algorithm23}.
	\begin{itemize}
		\item $\Delta=\langle 8,9,10,11,12,13,14,15\rangle,$ $\Sat(7)=\{\Delta\}$ and $B=\{\Delta\}.$
		\item By Proposition \ref{proposition7}, we know that $\Ap(\Delta,8)=\{0,9,10,11,12,13,14,15\}.$ By using Remark \ref{remark11}, we have that $\SG(\Delta)=\{4,5,6,7\}$ and by using Proposition \ref{proposition19}, $\theta(\Delta)=\{4,5,6\}.$
		\item $C=\{\Delta\cup \{4\}, \Delta\cup \{5\}, \Delta\cup \{6\} \}$ and by applying Proposition \ref{proposition21}, we have that $\msg(\Delta\cup \{4\})=\{4,9,10,11\}, $ $\msg(\Delta\cup \{5\})=\{5,8,9,11,12\} $ and $\msg(\Delta\cup \{6\})=\{6,8,9,10,11,13\}. $		
		\item $\Sat(7)=\{\Delta,\Delta\cup \{4\}, \Delta\cup \{5\}, \Delta\cup \{6\}\}$ and $B=\{\Delta\cup \{4\}, \Delta\cup \{5\}, \Delta\cup \{6\}\}.$
		\item $\Ap(\Delta\cup \{4\},4)=\{0,9,10,11\},$ $\Ap(\Delta\cup \{5\},5)=\{0,8,9,11,12\}$ and $\Ap(\Delta\cup \{6\},6)=\{0,8,9,10,11,13\}.$ Then $\SG(\Delta \cup \{4\})=\{5,6,7\}, $ $\SG(\Delta \cup \{5\})=\{4,6,7\} $ and  $\SG(\Delta \cup \{6\})=\{3,4,5,7\}. $ Therefore, 
		$\theta(\Delta \cup \{4\})=\emptyset=\theta(\Delta \cup \{5\})$ and  $\theta(\Delta \cup \{6\})=\{3,4\}. $
		\item $C=\{\Delta\cup \{3,6\}, \Delta\cup \{4,6\} \},$ $\msg(\Delta\cup \{3,6\})=\{3,8,10\} $ and  $\msg(\Delta\cup \{4,6\})=\{4,6,9,11\} .$ 
		
		\item $\Sat(7)=\{\Delta,\Delta\cup \{4\}, \Delta\cup \{5\}, \Delta\cup \{6\}, \Delta\cup \{3,6\}, \Delta\cup \{4,6\}\}$ and $B=\{\Delta\cup \{3,6\}, \Delta\cup \{4,6\}\}.$
		
		\item $\Ap(\Delta\cup \{3,6\},3)=\{0,8,10\}$ and  $\Ap(\Delta\cup \{4,6\},4)=\{0,6,9,11\}.$ Then $\SG(\Delta \cup \{3,6\})=\{5,7\}$ and   $\SG(\Delta \cup \{4,6\})=\{2,5,7\}.$  Therefore, 
		$\theta(\Delta \cup \{3,6\})=\emptyset$ and $\theta(\Delta \cup \{4,6\})=\{2\}.$	
		\item $C=\{\Delta\cup \{2,4,6\}\}$ and $\msg(\Delta\cup \{2,4,6\})=\{2,9\}. $ 
		
		\item $\Sat(7)=\{\Delta,\Delta\cup \{4\}, \Delta\cup \{5\}, \Delta\cup \{6\}, \Delta\cup \{3,6\}, \Delta\cup \{4,6\}, \Delta\cup \{2, 4,6\}\}$ and $B=\{\Delta\cup \{2,4,6\}\}.$
		
		\item $\Ap(\Delta\cup \{2,4,6\},2)=\{0,9\}.$  Then $\SG(\Delta \cup \{2,4,6\})=\{7\}$ and   
		$\theta(\Delta \cup \{2,4,6\})=\emptyset.$

		\item The algorithm returns $$\Sat(7)=\{\Delta,\Delta\cup \{4\}, \Delta\cup \{5\}, \Delta\cup \{6\}, \Delta\cup \{3,6\}, \Delta\cup \{4,6\}, \Delta\cup \{2, 4,6\} \}.$$
	\end{itemize}
\end{example}
\section{The elements of $\Sat(F)$ with fixed genus}
Given $F$ and $g$  positve integers, denote by 
$$
\Sat(F,g)=\{S\in \Sat(F)\mid \g(S)=g\}.
$$
From Proposition \ref{proposition13} it is deduced the following result.

\begin{lemma}\label{lemma25}
With the previous notation, if $\Sat(F,g)\neq \emptyset,$	then $\frac{F+1}{2}\leq g\leq F.$	
\end{lemma}

Let $S$ be a numerical semigroup, then the {\it   associated sequence} to $S$ is recurrently defined   as follows:
\begin{itemize}
\item $S_0=S,$
\item $S_{n+1}=S_n\backslash \{\m(S_n)\}$ for all $n\in \N.$	
	\end{itemize}

Let $S$ be a numerical semigroup. We say that an element $s$ of $S$ is an {\it small elemnent} of $S$ if $s<\F(S).$ Denote by $\NN(S)$ the set of small elements of $S.$ The cardinality of $\NN(S)$ will be denoted by $\n(S).$

It is clear that  the disjoint union of the sets $\NN(S)$ and $\N\backslash S$ is the set $\{0,\dots, \F(S)\}.$ Therefore, we have the following result.
\begin{lemma}\label{lemma26}
	If $S$ is a numerical semmigroup, then $\g(S)+\n(S)=\F(S)+1.$	
\end{lemma}
If $S$ is a numerical semigroup and  $\{S_n\}_{n\in \N}$ is its associated sequence, then the set  $\Cad(S)=\{S_0,S_1,\dots, S_{\n(S)-1}\}$ is called the {\it associated chain} to $S.$ Observe that 
$S_0=S$ and $S_{\n(S)-1}=\Delta(\F(S)+1).$

Observe that, from Proposition \ref{proposition16}, we know that if $S\in \Sat(F),$ then $\Cad(S)\subseteq \Sat(F).$ Therefore, we can enounce the following result.

\begin{lemma}\label{lemma27} If $S\in \Sat(F),$ then $\Sat(F,g)\neq \emptyset$ for all $g\in \{\g(S),\cdots,F\}.$
\end{lemma}
Our next aim is to determine the minimum element of the set $\{\g(S)\mid S\in \Sat(F)\}.$ For this purpose we introduce the following notation. Let $\{a,b\}\subseteq \N,$ then we denote by
$$
\T(a,b)=\langle a \rangle\cup \{x\in \N\mid x\ge b\}.
$$

For integers $a$ and $b,$ we say that $a$ {\it divides} $b$ if there exists an integer $c$ such that $b=ca,$ and we denote this by $a\mid b.$ Otherwise, $a$ {\it does not divide} $b$, and we denote this by $a\nmid b.$

The nex lemma is \cite[Lemma 2.3]{houston2} and it shows a characterization  of saturated numerical semigroups.

\begin{lemma}\label{lemma28}
Let $S$ be a numerical semigroup. Then 	$S$ is a saturated numerical semigroup if and only if there are positive integers $a_1,b_1,\cdots,a_n,b_n$ verifying the following properties:
\begin{enumerate}
	\item[1)] $a_{i+1}\mid a_i$ for all $i\in \{1,\cdots, n-1\}.$
	\item[1)] $a_i<b_i<b_{i+1}$ for all $i\in \{1,\cdots, n-1\}.$
	\item[1)] $S=\T(a_1,b_1)\cap \cdots \cap \T(a_n,b_n) .$	
\end{enumerate}
\end{lemma}

The next Lemma is an immediate consequence of Lemma \ref{lemma28}. 
\begin{lemma}\label{lemma29}
	If $S$ is a maximal element of $\Sat(F),$ then $S=\T(a,F+1)$ for some $a\in \{1,\cdots, F\}$ such that $a\nmid F.$	
\end{lemma}

If $n$ is a positive integer, then we denote $\A(n)=\{x\in \{1,\cdots, n\}\mid x\nmid n\}$ and $\B(n)=\{x\in \A(n)\mid x'\nmid x \mbox{ for all }x'\in \A(n)\backslash \{x\}\}.$\\

The following result is a consequence of Lemmas \ref{lemma28} and \ref{lemma29}.

\begin{theorem}\label{theorem30}
	With the previous notation, $S$ is a maximal element of $\Sat(F)$ if and only if $S=\T(x,F+1)$ for some $x\in \B(F).$	
\end{theorem}

In the following Example, we illustrate how the previous theorem works. 
\begin{example}\label{example31}
We are going to apply the Theorem \ref{theorem30} to compute the maximal elements of $\Sat(30).$ 	As $$\A(30)=\{4,7,8,9,11,12,13,14,16,17,18,19,20,21,22,23,24,25,26,27,28,29\},$$ then $\B(30)=\{4,7,9,11,13,17,19,23,25,29\}.$ Therefore, by applying Theorem \ref{theorem30}, we have that the set formed by the maximals elements of $\Sat(30)$ is $\{ \T(4,31),\, \T(7,31),\, \T(9,31),\, \T(11,31),\, \T(13,31),\, \T(17,31),\, \T(19,31),\, \T(23,31),\\ \T(25,31),\,\T(29,31)\}.$
\end{example}

Let $q\in \Q.$ Denote $\lfloor q \rfloor=\max\{z\in \Z\mid z\leq q\}.$ 
The following result is a consequence of Theorem \ref{theorem30}.
\begin{corollary}\label{corollary32}
	If $p$ is the least positive integer such that $p\nmid F$,
 then $\min\{\g(S)\mid S\in \Sat(F)\}=F-\displaystyle \left\lfloor \frac{F}{p}\right\rfloor.$
	\end{corollary}
By using this corollary, in the following example we calculate the minimum genus of the elements belonging  to $\Sat(7),$ as well as the minimum genus of the elementos of  $\Sat(6).$

\begin{example} \label{example33}
	We have that
	\begin{itemize}
		\item minimum$\{\g(S)\mid S\in \Sat(7)\}=7- \left\lfloor \frac{7}{2}\right\rfloor=7-3=4.$ Moreover, 
		$\g(\T(2,8))=4.$
		\item minimum$\{\g(S)\mid S\in \Sat(6)\}=6- \left\lfloor \frac{6}{4}\right\rfloor=6-1=5.$ In addition, $\g(\T(4,7))=5.$
	\end{itemize}
	\end{example}
We have now all the ingredients needed to present the announced algorithm.

\begin{algorithm}\label{algorithm34}\mbox{}\par
\end{algorithm}
\noindent\textsc{Input}: Two positive integers  $F$ and $g$ such that  $\frac{F+1}{2}\leq g\leq  F.$	  \par
\noindent\textsc{Output}: $\Sat(F,g).$

\begin{enumerate}
	\item[(1)] Compute the  smallest  positive integer $p$ such that $p\nmid F.$
	\item[(2)] If $g<F-\displaystyle \left\lfloor \frac{F}{p}\right\rfloor,$ then return $\emptyset.$
	\item[(3)] $\Delta=\langle F+1,\cdots, 2F+1\rangle,$ $H=\{\Delta\},$ $i=F.$
	\item[(4)] If $i=g,$ then return $H.$		
	\item[(5)] For all $S\in H,$ compute the set $\theta(S)=\{x\in \SG(S)\mid x<\m(S), x\neq F \mbox{ and } S\cup \{x\} \mbox{ is a saturated numerical semigroup}\}.$ 
	\item[(6)]  $H=\bigcup_{S\in H}\{S\cup \{x\}\mid x\in \theta(S)\},$ $i=i-1$ and    go to Step $(4).$ 	
\end{enumerate}
Next we illustrate this algorithm with an example.

\begin{example}\label{example35}
	By using the Algorithm \ref{algorithm34}, we are going to calculate the set $\Sat(7,5).$
	\begin{itemize}
		\item $2$ is the smallest positive integer such that it does not divide to $7$ and $7- \left\lfloor \frac{7}{2}\right\rfloor=7-3=4< 5,$ therefore we can assert that $\Sat(7,5)\neq \emptyset.$
		\item $\Delta=\langle 8,9,10,11,12,13,14,15\rangle,$ $H=\{\Delta\},$ $i=7.$
		\item $\theta(\Delta)=\{4,5,6\}.$
		\item $H=\{\Delta \cup \{4\}, \Delta \cup \{5\},\Delta \cup \{6\}\},$ $i=6.$
		\item $\theta(\Delta \cup \{4\})=\emptyset,$ $\theta(\Delta \cup \{5\})=\emptyset$ and $\theta(\Delta \cup \{6\})=\{3,4\}.$
		\item $H=\{\Delta \cup \{3,6\}, \Delta \cup \{4,6\}\},$ $i=5.$
		\item The algorithm returns $\{\Delta \cup \{3,6\}, \Delta \cup \{4,6\}\}.$
	\end{itemize}

\end{example}

\section{$\Sat(F)$-system of generators}
We will say that a set $X$ is a $\Sat(F)$-{\it set} if it verifies the following conditions:
\begin{enumerate}
	\item[1)]$X\cap \Delta(F+1)=\emptyset.$
	\item[2)]There is $S\in \Sat(F)$ such that $X\subseteq S.$	
\end{enumerate} 

If $X$ is a $\Sat(F)$-set, then the intersection of all elements of $\Sat(F)$ containing $X$ will be denoted by $\Sat(F)[X].$  As $\Sat(F)$ is a finite set, then by applying  Proposition \ref{proposition16}, we have that the intersection of elements of $\Sat(F)$ is  again an element of $\Sat(F).$ Therefore, we have the following result.
\begin{proposition}\label{proposition36}
	Let $X$ be a $\Sat(F)$-set. Then $\Sat(F)[X]$ is the smallest element of $\Sat(F)$ containing $X.$
\end{proposition}

If $X$ is a $\Sat(F)$-set and $S=\Sat(F)[X],$ we will say that $X$ is a $\Sat(F)$-{\it system of generators} of $S.$ Besides, if $S\neq\Sat(F)[Y]$ for all $Y\subsetneq X,$ then $X$ is a {\it minimal} $\Sat(F)$-{\it system of generators} of $S.$

Our next aim in this section will be  to prove that every element of   $\Sat(F)$ has a unique minimal $\Sat(F)$-system of generators. \\

The following result appears in \cite[Lemma 8]{houston1}. 

\begin{lemma}\label{lemma37}
	Let $S$ be a saturated numerical semigroup and $x\in S\backslash \{0\}.$ Then the following conditions are equivalent.
	\begin{enumerate}
		\item[1)]$S\backslash \{x\}$ is a saturated numerical semigroup.
		\item[2)]If $y\in S\backslash    \{0\}$ and $y<x,$ then $\d_S(y)\neq d_S(x).$
	\end{enumerate}
	
\end{lemma}
\begin{lemma}\label{lemma38} Let $S\in \Sat(F)$ and $s\in S$ such that $0<s<F$ and $\d_S(s)\neq \d_S(s')$ for all $s'\in S$ being $0<s'<s.$ If $X$ is a $\Sat(F)$-system of generators of $S,$ then $s\in X.$	
\end{lemma}
\begin{proof}
	By Lemma \ref{lemma37}, we know that $S\backslash \{s\}$ is an element of $\Sat(F).$ If $s\notin X,$ then $X\subseteq S\backslash \{s\}$ and, by applying Proposition \ref{proposition36}, we have that $S=\Sat(F)[X]\subseteq S\backslash \{s\},$ which is  absurd. 
\end{proof}

The following result can be consulted in \cite[Theorem 4]{houston1}.
\begin{lemma}\label{lemma39}	Let $A \subseteq \N$ such that $0\in A$ and $\gcd(A)=1.$ Then the following conditions are equivalent.
	\begin{enumerate}
		\item[1)]$A$ is a saturated numerical semigroup.
		\item[2)]$a+\d_A(a)\in A$ for all $a\in A.$
		\item[3)]$a+k\cdot \d_A(a)\in A$ for all $(a,k)\in A\times \N.$		
	\end{enumerate}	
\end{lemma}

\begin{lemma}\label{lemma40} Let $S\in \Sat(F)$ and $X=\{x\in S\backslash \{0\}\mid \d_S(x)\neq \d_S(y) \mbox{ for all } y\in S \mbox{ with }y<x \mbox{ and }x<F\}.$ Then $\Sat(F)[X]=S.$	
\end{lemma}
\begin{proof}
	Let $T=\Sat(F)[X].$ As $X\subseteq S,$ then by applying Proposition \ref{proposition36}, we have that $T\subseteq S.$ Now we will see the reverse inclusion, that is, $S\subseteq T.$ Assume that $X=\{x_1,\dots, x_n\},$ $s\in S\backslash \{0\}$ and $x_1<\cdots<x_k\leq s<x_{k+1}<\cdots<x_r.$ Then $\d_S(s)=\d_S(x_k)=\d_T(x_k)$ and $s=x_k+a$ for some $a\in \N.$ We deduce then $\d_S(x_k)\mid a$ and so $s=x_k+t\cdot \d_S(x_k)$ for some $t\in \N.$ Consequently, by applying Lemma \ref{lemma39}, $s=x_k+t\cdot \d_T(x_k)\in T.$
\end{proof}

As a consequence of Lemmas \ref{lemma38} y \ref{lemma40}, we can assert that the minimal $\Sat(F)$-system of generators is unique. This is the content of the following proposition.
\begin{proposition}\label{proposition41}
If $S\in \Sat(F),$ then the unique minimal $\Sat(F)$-system of generators of $S$ is the set
$$
\{x\in S\backslash \{0\}\mid x<F \mbox{ and }\d_S(x)\neq \d_S(y) \mbox{ for all } y\in S \mbox{ such that }y<x\}.
$$
\end{proposition}

If $S\in\Sat(F),$ then we denote by $\Sat(F)\msg(S)$ the minimal $\Sat(F)$-system of generators of $S.$ The cardinality of  $\Sat(F)\msg(S)$  is called the $\Sat(F)$-{\it rank} of $S$ and it will denote by $\Sat(F)$-${\rank}(S).$ Let us illustrate these two concepts with an example.

\begin{example}\label{example42}
	It is clear that $S=\{0,4,8,10,12,14,16,18,20,22,\rightarrow\}\in \Sat(21).$ By applying Proposition \ref{proposition41}, we obtain that $\Sat(21)\msg(S)=\{4,10\}.$ Therefore, $\Sat(21)$-${\rank}(S)=2.$
	
\end{example}

By applying Lemma \ref{lemma39}, we can easily prove  the following result. 
\begin{lemma}\label{lemma43}
	Let $n_1<n_2<\cdots<n_p<F$  be positive integers, $d=\gcd\{n_1,\cdots,\\
	n_p\}$ and $d\nmid F.$ For every $i\in \{1,\cdots, p\},$ let $d_i=\gcd\{n_1,\cdots,n_i\},$ for each $j\in \{1,\cdots,p-1\},$ let $k_j=\max\{k\in \N\mid n_j+kd_j<n_{j+1}\}$ and $k_p=\max\{k\in \N\mid n_p+kd_p<F\}.$ Then $\Sat(F)[\{n_1,\cdots,n_p\}]=\{0,n_1,n_1+d_1,\cdots,n_1+k_1d_1, n_2,n_2+d_2,\cdots,n_2+k_2d_2,\cdots, n_{p-1},n_{p-1}+d_{p-1},\cdots,n_{p-1}+k_{p-1}d_{p-1},n_p,n_p+d_p,\cdots,n_p+k_pd_p,F+1,\rightarrow\}.$
\end{lemma}
As a consequence of Proposition \ref{proposition41} and Lemma \ref{lemma43}, we present in the following proposition a characterization of the minimal $\Sat(F)$-system of generators of $\Sat(F)[\{n_1,\cdots,n_p\}].$
\begin{proposition}\label{proposition44}
	Let $n_1<n_2<\cdots<n_p<F$ be positive integers, $d=\gcd\{n_1,\cdots,n_p\}$ and $d\nmid F.$ Then  $\{n_1,\cdots,n_p\}$ is the minimal $\Sat(F)$-system of generators of $\Sat(F)[\{n_1,\cdots,n_p\}]$ if and only if $\gcd\{n_1,\cdots,n_i\}\neq \gcd\{n_1,\cdots,\\n_{i+1}\}$ for all $i\in \{1,\cdots, p-1\}.$
\end{proposition}
\begin{example}\label{example45}
	By applying Lemma \ref{lemma43}, we deduce that  $\Sat(51)[\{18,28,42\}]=\{0,8,16,24,28,32,36,40,42,44,46,48,50,52,\rightarrow\}.$ Moreover, as $\gcd\{8\}>\\ \gcd\{8,28\}>\gcd\{8,28,42\},$ then, by Proposition \ref{proposition44}, we know that 
	$\{8,28,42\}$ is the minimal $\Sat(51)$-system of generators of $\Sat(51)[\{8,28,42\}].$
\end{example}
The following result is a direct consequence of Proposition \ref{proposition41}.
\begin{lemma}\label{lemma46}
	If $S\in \Sat(F)$ and $S\neq \Delta(F+1),$ then $\m(S)\in \Sat(F)\msg(S).$
\end{lemma}
By applying Proposition \ref{proposition12}, Lemma \ref{lemma37} and Proposition \ref{proposition41}, we have the following result. 

\begin{proposition}\label{proposition47}
	If $S\in\Sat(F),$ then the following conditions are verified.
	\begin{enumerate}
		\item[1)] $\Sat(F)$-${\rank}(S)\leq \e(S).$
		\item[2)] $\Sat(F)$-${\rank}(S)=0$ if and only if $S=\Delta(F+1).$
		\item[3)] $\Sat(F)$-${\rank}(S)=1$ if and only if $\Sat(F)\msg(S)=\{\m(S)\}.$		
	\end{enumerate}
\end{proposition}

As a consequence of Proposition \ref{proposition47} and Lemma \ref{lemma43}, we have the following result.
\begin{corollary}\label{corollary48} Under the standing notation, the following conditions are equivalent.
	\begin{enumerate}
		\item[1)] $S\in \Sat(F)$ and $\Sat(F)$-${\rank}(S)=1.$
		\item[2)] There is $m\in \N$ such that $2\leq m <F,$  $m\nmid F$ and $S=T(m,F+1).$
	\end{enumerate}
\end{corollary}

\section{$\Sat(F)$-sequences}
Give $k\in \N\backslash \{0\}$, a $\Sat(F)$-{\it sequence, of lenght} $k,$ is a  $k$-sequence of positive integers $(d_1,d_2,\dots,d_k)$ such that $d_1>d_2>\dots>d_k,$ such that $d_{i+1}\mid d_i$ for all $i\in \{1,\cdots, k-1\}$ and $d_k\nmid F.$ 
\begin{theorem}\label{theorem49}
If $(d_1,d_2,\dots, d_p)$ is a $\Sat(F)$-sequence and $t_1,t_2,\cdots, t_p$ are positive integers such that $t_1d_1+\cdots+t_pd_p<F$ and $\gcd \displaystyle \left\{\frac{d_i}{d_{i+1}}, t_{i+1}\right\}=1$ for all $i\in \{1,\dots, p-1\},$ then $\{d_1,t_1d_1+t_2d_2,\cdots,t_1d_1+t_2d_2+\cdots+t_pd_p\}$ is the minimal $\Sat(F)$-system of generators of an element of $\Sat(F)$ with 	$\Sat(F)$-rank equal to $p.$ Moreover, every minimal $\Sat(F)$-system of generators of an element of $\Sat(F)$ with 	$\Sat(F)$-rank equal to $p,$ has this form.
\end{theorem}

\begin{proof}
	It is easy to see that $\gcd \{d_1,t_1d_1+t_2d_2,\cdots,t_1d_1+t_2d_2+\cdots+t_id_i\}=d_i$ for all $i\in \{1,\cdots,p\}.$ By aplying Proposition \ref{proposition44}, we obtain that $\{d_1,t_1d_1+t_2d_2,\cdots,t_1d_1+t_2d_2+\cdots+t_pd_p\}$ is the minimal $\Sat(F)$-system of generators of an element of $\Sat(F)$ with 	$\Sat(F)$-rank equal to $p.$
	
	Conversely, if $\{n_1<n_2<\dots <n_p\}$ is the minimal  $\Sat(F)$-system of generators of an element of $\Sat(F)$ and $d_i=\gcd\{n_1,\cdots,n_i\}$ for all $i\in \{1,\cdots,p\},$ then by applying Proposition \ref{proposition44} and Lemma \ref{lemma43}, we have that $(d_1,\dots, d_p)$ is a $\Sat(F)$-sequence. To conclude the proof, we will see that there are positive integers $t_1,\cdots, t_p$ such that $n_1=d_1,$ $n_2=t_1d_1+t_2d_2,\cdots, n_p=t_1d_1+t_2d_2+\cdots+t_pd_p$ and $\gcd \displaystyle \left\{\frac{d_i}{d_{i+1}}, t_{i+1}\right\}=1$ for all $i\in \{1,\dots, p-1\}.$
	Let $t_1=1$ and $t_{i+1}=\frac{n_{i+1}-n_i}{d_{i+1}}$ for all $i\in \{1,\cdots,p-1\}.$  Let us  prove,  by induction on $i,$ that $n_i=t_1d_1+\cdots+t_id_i$ for all $i\in \{2,\dots, p\}.$ For $i=2,$ the result is true since $t_1d_1+t_2d_2=1\cdot n_1+\frac{n_2-n_1}{d_2}d_2=n_2.$ As $n_{i+1}=n_i+t_{i+1}d_{i+1},$ then by induction hypothesis, we have  $n_{i+1}=t_1d_1+\cdots+t_id_i+t_{i+1}d_{i+1}.$ To conclude the proof, it suffices to show that $\gcd \displaystyle \left\{\frac{d_i}{d_{i+1}}, t_{i+1}\right\}=1$ for all $i\in \{1,\dots, p-1\}.$ In fact, $d_{i+1}=\gcd\{n_1,\cdots,n_{i+1}\}=\gcd\{\gcd\{n_1,\cdots,n_i\},n_{i+1}\}=\gcd\{d_i,t_1d_1+\cdots+t_id_i+t_{i+1}d_{i+1} \}=\gcd\{d_i,t_{i+1}d_{i+1} \}=d_{i+1}\cdot\gcd \displaystyle \left\{\frac{d_i}{d_{i+1}}, t_{i+1}\right\}.$ Therefore, $\gcd \displaystyle \left\{\frac{d_i}{d_{i+1}}, t_{i+1}\right\}=1.$
\end{proof}
As a direct consequence of the previous theorem, we have the following result.
\begin{corollary}\label{corollary50}If $(d_1,d_2,\dots, d_p)$ is a $\Sat(F)$-sequence and  $d_1+d_2+\cdots+d_p<F,$  then $\{d_1,d_1+d_2,\cdots,d_1+d_2+\cdots+d_p\}$ is the minimal $\Sat(F)$-system of generators of an element of $\Sat(F).$
\end{corollary}

As a consequence of Theorem \ref{theorem49} and Corollary \ref{corollary50}, if we want to compute all the elements belonging to $\Sat(F)$ with  	$\Sat(F)$-rank equal to $p$, it will be enough to do the following steps:
\begin{enumerate}
	\item[(1)] To compute $$\L(F,p)=\{(d_1,\dots, d_p)\mid (d_1,\dots, d_p) \mbox{ is a }  \Sat(F)\mbox{-sequence and }$$
	$d_1+\dots+ d_p<F\}.$
	\item[(2)] For every $(d_1,\dots, d_p)\in \L(F,p),$ compute
	$$
	\C(d_1,\dots, d_p)=\{(t_1,\dots, t_p)\in (\N\backslash\{0\})^p\mid t_1d_1+\cdots+t_pd_p<F \mbox{ and }$$
	$\gcd \displaystyle \left\{\frac{d_i}{d_{i+1}}, t_{i+1}\right\}=1\mbox{ for all }i\in \{1,\dots, p-1\}\}.
	$	
\end{enumerate}
\begin{proposition}\label{proposition51}If $\{a_1,a_2,\cdots,a_p\}\subseteq \N\backslash \{0,1\}$ and $a_1\nmid F,$ then $(a_1a_2\cdots a_p, \\
	a_1a_2\cdots a_{p-1},\cdots,a_1)$ is a $\Sat(F)$-sequence of length $p.$ Moreover, every $\Sat(F)$-sequence of length $p$ is of this form.
\end{proposition}
\begin{corollary}\label{corollary52}
	Let $a$ be the smallest positive integer that does not divide $F.$ Then $\Sat(F)$  contains at least one element of 	$\Sat(F)$-rank equal to $p$ if and only if $a(2^p-1)<F.$
\end{corollary} 
\begin{proof}
	By applying Theorem \ref{theorem49} and Corollary \ref{corollary50}, we deduce that $\Sat(F)$ contains at least an element of $\Sat(F)$-rank equal to $p$ if and only if $\L(F,p)\neq \emptyset.$	
	By applying now Proposition \ref{proposition51}, we have that $\L(F,p)\neq \emptyset.$ if and only if there is $\{a_1,a_2,\cdots, a_p \}\subseteq \N\backslash \{0,1\}$ such that $a_1\nmid F$ and $a_1a_2\cdots a_p+a_1a_2\cdots a_{p-1}+\cdots+a_1<F.$ To conclude the proof, it suffices to note that  this  is verified if and only if $a\cdot 2^{p-1}+a\cdot 2^{p-2}+\cdots+a<F.$ By using the formula of the sum of a geometry progression, we obtain that $a\cdot 2^{p-1}+a\cdot 2^{p-2}+\cdots+a<F$ if and only if $a(2^p-1)<F.$
	
\end{proof}
\begin{example}\label{example53} We can assert, by using Corollary \ref{corollary52}, that $\Sat(18)$ does not have  elements with $\Sat(F)$-rank equal to $3,$ because $4(2^3-1)>18.$
\end{example}

We finish this work showing an algorithm which allows us to compute the set $\C(d_1,
\dots, d_p)$ from $(d_1,
\dots, d_p)\in \L(F,p).$

In first time we note that to compute the set
$$
\{(t_1,\dots, t_p)\in (\N\backslash\{0\})^p\mid t_1d_1+\cdots+t_pd_p\leq F-1\}
$$
is equivalent to compute the set
$$
\{(x_1,\dots, x_p)\in \N^p\mid d_1x_1+\cdots+d_px_p\leq F-1-(d_1+\dots+ d_p)\}.
$$

Also, observe that
$$
\{(x_1,\dots, x_p)\in \N^p\mid d_1x_1+\cdots+d_px_p\leq F-1-(d_1+\dots+ d_p)\}=
$$
$$
\{(x_1,\dots, x_p)\in \N^p\mid d_1x_1+\cdots+d_px_p=k \mbox { for some }$$
$$k\in \{0,\cdots, F-1-(d_1+\dots+ d_p)\}.
$$

If $(x_1,\dots, x_p)\in \N^p$ and  $d_1x_1+\cdots+d_px_p=k, $ then $d_p\nmid k.$ Hence, $k=a\cdot d_p$ and consequently, $\{(x_1,\dots, x_p)\in \N^p\mid d_1x_1+\cdots+d_px_p=k\}=\{(x_1,\dots, x_p)\in \N^p\mid \frac{d_1}{d_p}x_1+\cdots+\frac{d_p}{d_p}x_p=a \}.$

Finally, observe that Algorithm 14 from \cite{AMC}, allows us compute the set $\{(x_1,\dots, x_p)\in \N^p\mid \frac{d_1}{d_p}x_1+\cdots+\frac{d_p}{d_p}x_p=a \}.$

\begin{algorithm}\label{algorithm54}\mbox{}\par
\end{algorithm}
\noindent\textsc{Input}: $(d_1,\dots, d_p)\in \L(F,p).$ \par
\noindent\textsc{Output}: $\C(d_1,\dots, d_p).$

\begin{enumerate}
	\item[(1)] $\alpha=F-1-(d_1+\dots+ d_p).$
	\item[(2)] For all $k\in \{0,\cdots, \lfloor \frac{\alpha}{d_k}\rfloor \},$ by using Algorithm 14 from \cite{AMC}, compute $D_k=\{(x_1,\dots, x_p)\in \N^p\mid \frac{d_1}{d_p}x_1+\cdots+\frac{d_p}{d_p}x_p=k \}.$
	\item[(3)] For all $k\in \{0,\cdots, \lfloor \frac{\alpha}{d_k}\rfloor \},$ let $E_k=\{(x_1+1,\dots, x_p+1)\mid (x_1,\dots, x_p)\in D_k \}.$
	\item[(4)] 	$A=\bigcup_{k=0}^{\lfloor \frac{\alpha}{d_k}\rfloor}E_k.$	
	\item[(5)] Return $\{(t_1,\cdots,t_p)\in A\mid \gcd \displaystyle \left\{\frac{d_i}{d_{i+1}}, t_{i+1}\right\}=1  \mbox{ for all } i\in \{1,\dots, p-1\}\}.$
\end{enumerate}

\end{document}